\documentclass[12pt]{article}
\usepackage[usenames]{color}
\usepackage{graphicx, subfigure}
\usepackage{amsthm}
\usepackage{amsmath}
\usepackage{amsfonts}

\newtheorem{theorem}{Theorem}[section]

 \newtheorem{corollary}{Corollary}[section]

\title{On the relationship between three classes of operators on Riesz spaces}

\author{Liang Hong \\
Department of Mathematics \\
Robert Morris University   \\
Moon Township, PA 15108, USA}

\begin{document}

\maketitle

\begin{abstract}
Following several papers in the prior literature, we study the relationship between order bounded operators, topologically bounded operators and topologically continuous operators. Our main contribution is two folded: (i) we provide a set of counterexamples to illustrate several extant results in the literature; (ii) we give conditions for the space of order bounded operators to coincide with the space of topologically bounded operators as well as conditions for these two spaces to coincide with the space of topologically continuous operators. \\

\noindent {\it 2010 MSC:} Primary 47B60, 47B65; Secondary 46A40, 06B30, 06F30.\\

\noindent \emph{Keywords and phrases:} Counterexamples; locally solid Riesz spaces; order bounded operators; regular operators; topologically bounded operators; topologically continuous operators.
\end{abstract}

\section{Introduction}
The relationship between order bounded operators and order continuous operators on Riesz spaces has been investigated in \cite{A1}, \cite{AB3} and \cite{Schaefer1}. This paper aims to further study the relationship between order bounded operators, topologically bounded operators and topologically continuous operators on Riesz spaces. The first part of our paper provides several counterexamples to illustrate a few existing results along this line. We first ask whether the space $\mathcal{L}_{tc}(E_1, E_2)$ of topologically continuous operators is a vector subspace of the space $\mathcal{L}_{b}(E_1, E_2)$ of order bounded operators, or whether the converse is true. We show by counterexamples that the answer to both is negative. Theorem~2.3 in \cite{Hong} gave a sufficient condition for $\mathcal{L}_{tc}(E_1, E_2)$ to be a vector subspace of $\mathcal{L}_{b}(E_1, E_2)$. We give a counterexample to show that neither of the hypotheses ``The topology on the image space is locally solid.'' and ``The image space has an order bounded neighborhood around zero.'' of that theorem may be dropped. In addition, we argue that this sufficient condition is not necessary. The Nakano-Roberts theorem says that the topological dual of a locally solid Riesz space is an ideal of the order dual. We provide a counterexample to show that the condition of Theorem~2.3 in \cite{Hong} is not sufficient to extend the Nakano-Roberts theorem to the case where the image space is an ordered topological space, i.e., $\mathcal{L}_{tc}(E_1, E_2)$ may not be an ideal of $\mathcal{L}_{b}(E_1, E_2)$ even if $E_1$ is equipped with a locally solid topology.


In the second part of the paper, we seek conditions for the space $\mathcal{L}_{b}(E_1, E_2)$ to coincide with the space $\mathcal{L}_{tb}(E_1, E_2)$ of topologically bounded operators and conditions for these two spaces to coincide with $\mathcal{L}_{tc}(E_1, E_2)$. These results combined together yield conditions for $\mathcal{L}_{b}(E_1, E_2)$, the space $\mathcal{L}_{r}(E_1, E_2)$ of regular operators, $\mathcal{L}_{tb}(E_1, E_2)$ and $\mathcal{L}_{tc}(E_1, E_2)$ to coincide at the same time. This further leads to several interesting results: under suitable conditions, each positive operator is topologically continuous, each topologically continuous can be written as the difference of two positive operators, and so on.

The remainder of the paper is organized as follows. Section~2 provides the aforementioned counterexamples. Section~3 investigates the relationship between $\mathcal{L}_{b}(E_1, E_2)$, $\mathcal{L}_{tb}(E_1, E_2)$ and $\mathcal{L}_{tc}(E_1, E_2)$. For notation, terminology and standard results concerning topological vector spaces, we refer to \cite{Bourbaki1}, \cite{Horvath}, \cite{NB} and \cite{Schaefer}; for notation, terminology and standard results concerning Riesz spaces and operators on them, we refer to \cite{AB2}, \cite{AB1}, \cite{LZ} and \cite{Zaanen}.


\section{Counterexamples}
To set the stage, we first recall some results about the relationship between order bounded sets and topologically bounded sets in ordered topological vector spaces. A classical result says that an order bounded subset of a locally solid Riesz space must be topologically bounded.

\begin{theorem}[Theorem 2.19 of \cite{AB1}]\label{theorem1.1}
If $(E, \tau)$ is a locally solid Riesz space, then every order bounded subset of $E$ is $\tau$-bounded.
\end{theorem}

Recently, \cite{Hong} showed that if an order bounded topological vector space has an order bounded neighborhood of zero, then every topologically bounded subset must be order bounded.

\begin{theorem}[Theorem 2.4 of \cite{Hong}]\label{theorem1.2}
Let $(E, \tau)$ be an order topological vector space that has an order bounded $\tau$-neighborhood of zero. Then very $\tau$-bounded subset of $E$ is order bounded.
\end{theorem}


\cite{Nakano1} and \cite{Nakano2} first showed that the norm dual of a normed Riesz space is an ideal of its order dual; later \cite{Roberts1} generalized this result to locally solid Riesz spaces but their terminology is of old-fashion. The following version, stated in the modern terminology of locally solid Riesz spaces, is taken from \cite{AB1}.

\begin{theorem}[Nakano-Roberts]\label{theorem2.1}
Let $(L, \tau)$ be a locally solid Riesz space. Then the topological dual $L'$ of $(L, \tau)$ is an ideal of the order dual $L^{\sim}$.
\end{theorem}

The Nakano-Roberts Theorem ensures that a continuous linear functional is necessarily order bounded. In view of this, one would naturally ask whether a topologically continuous operator is necessarily order bounded, that is, whether $\mathcal{L}_{tc}(E_1, E_2)\subset \mathcal{L}_b(E_1, E_2)$ holds. The next example shows that the answer is negative.\\

\noindent \textbf{Counterexample 3.1--A topologically continuous operator need not be order bounded.}\\
\noindent Let $E_1=E_2=\mathbb{R}^2$ and $\tau_1=\tau_2$ be the usual norm topology on $\mathbb{R}^2$. We equip $E_1$ and $E_2$ with the lexicographic ordering and the usual pointwise ordering, respectively. 
Take $x=(0, 0)$ and $y=(0, 1)$ in $E_1$. The order interval $B=[x, y]$ is order bounded in $E_1$, but not order bounded in $E_2$. Therefore, the identity operator is topologically continuously but not order bounded.\\



Example 3.1 shows that $\mathcal{L}_{tc}(E_1, E_2)\ \subset \mathcal{L}_{b}(E_1, E_2)$ generally does not hold. It is also natural to ask whether $\mathcal{L}_{b}(E_1, E_2)\subset \mathcal{L}_{tc}(E_1, E_2)$ holds, that is, whether an order bounded operator is necessarily topologically continuous. The answer is also negative as evidenced by the next example.\\


\noindent \textbf{Counterexample 3.2--An order bounded operator may not be topologically continuous.}  \\
\noindent Let $E_1=E_2$ be the space of all Lebesgue integrable functions on $\mathbb{R}$. Let  $\tau_1$ be the norm topology generated by the $L_1$-norm $||x||_1=\int_{\mathbb{R}} |x(t)|dt$ and $\tau_2$ be the weak topology $\sigma(E_2, E_2')$ on $E_2$. Equip both $E_1$ and $E_2$ with the ordering: $x\leq y$ if and only if $x(t)\leq y(t)$ for all $t\in \mathbb{R}$. Then the identity operator between $E_1$ and $E_2$ is order bounded. However, it is not topologically continuous because $\tau_2$ is strictly weaker than $\tau_1$.\\


In summary, we have concluded that
\begin{equation*}
\mathcal{L}_{tc}(E_1, E_2)\not \subset \mathcal{L}_{b}(E_1, E_2)\ \text{and $\mathcal{L}_{b}(E_1, E_2)\not \subset \mathcal{L}_{tc}(E_1, E_2)$}.
\end{equation*}

However, the following theorem  shows that if $E_1$ is locally solid and $E_2$ has an order bounded $\tau_2$-neighborhood, then  $\mathcal{L}_{tc}(E_1, E_2)$ is a vector subspace of $\mathcal{L}_b(E_1, E_2)$.\\

\begin{theorem}[Theorem 2.3 of \cite{Hong}]\label{theorem2.2}
Suppose $(E_1, \tau_1)$ is a locally solid Riesz space and $(E_2, \tau_2)$ is an ordered topological vector space having an order bounded $\tau_2$-neighborhood of zero. Then $\mathcal{L}_{tc}(E_1, E_2)$ is a vector subspace of $\mathcal{L}_b(E_1, E_2)$.
\end{theorem}


\noindent \textbf{Remark 1.} Example 2.1 shows that the hypothesis that ``$\tau_1$ is locally solid'' may not be dropped.\\


\noindent \textbf{Remark 2.} The hypothesis  ``$(E_2, \tau_2)$ has an order bounded $\tau_2$-neighborhood of zero'' cannot be dropped either even if both $\tau_1$ and $\tau_2$ are locally-convex solid. The following example illustrates this point. \\

\noindent \textbf{Counterexample 3.3--The hypothesis  ``$(E_2, \tau_2)$ has an order bounded $\tau_2$-neighborhood of zero'' in Theorem~\ref{theorem2.2} cannot be dropped.}\\
\noindent Let $E_1=E_2=D[-\pi, \pi]$ be the space of all the differentiable functions on $[-\pi, \pi]$. Take $\tau_1=\tau_2$ to be the norm topology generated by the sup norm $||x||_{\infty}=\sup_{-\pi\leq t\leq \pi} |x(t)|$.
Equip $E_1$ with the usual pointwise ordering; equip $E_2$ with the ordering defined as follows: for $x, y\in E_2$ we say $x\leq y$ if and only if $x(t)\leq y(t)$ and $x'(t)\leq y'(t)$ for all $t$ in $[-\pi, \pi]$.
Then  $\tau_1$ and $\tau_2$ are both locally convex-solid.
The identity operator is trivially topologically continuous. But it is not order bounded. To see this, take $B=\{\cos kt\}_{k\in N}$ in $E_1$ which is order bounded.
Suppose there exist two elements $x$ and $y$ in $E_2$ such that $B$ is contained in the order interval $[x, y]$ in $E_2$. Then we would have $k=\sup_{-\pi \leq t\leq \pi}|k\sin kt|\leq \sup_{-\pi \leq t\leq \pi}\max\{x(t), y(t)\}$ for all $k\in N$, which is impossible. Thus, $B$ is not order bounded in $E_2$. Since $B$ is contained in the open unit ball $U=\{x\in E_2\mid ||x||_{\infty}<1\}$ of $E_2$,
this also implies that $E_2$ does not have an order bounded $\tau_2$-neighborhood of zero.\\


We would also like to point out that the sufficient condition in Theorem \ref{theorem2.2} is not necessary. The next example shows that it is possible to have $\mathcal{L}_{tc}(E_1, E_2)\subset \mathcal{L}_b(E_1, E_2)$  with $\tau_1$ and $\tau_2$ both being locally-convex solid but $E_2$ does not have an order bounded $\tau_2$-bounded neighborhood of zero.\\

\noindent \textbf{Counterexample 3.4--The sufficient condition in Theorem~\ref{theorem2.2} is not necessary.}\\
\noindent Let $E_1=E_2$ be the space of all real-valued continuous functions defined on $\mathbb{R}$ and $\tau_1=\tau_2$ be the compact-open topology. Let $\mathcal{K}$ be the family of all compact subsets of $\mathbb{R}$. For each $K\subset \mathcal{K}$, we define $\rho_K(x)=\sup_{t\in K}|x(t)|$ for $x\in E_1=E_2$ and equip $E_1=E_2$ with the usual pointwise ordering.
Then $\tau_1=\tau_2$ is locally convex-solid.
Let $T\in \mathcal{L}_{tc}(E_1, E_2)$ and $B$ be an order bounded subset of $E_1$.
Then Theorem \ref{theorem1.1} implies that $T(B)$ is $\tau_2$-bounded. As a subset $B$ of $E_2$ is $\tau_2$-bounded if and only if $\rho_K(B)$ is bounded for every compact subset $K\in\mathcal{K}$ (p. 109 of \cite{Horvath}), we know $T(B)$ is $\tau_2$-bounded. However, $E_2$ cannot have an order bounded $\tau_2$-neighborhood of zero (Example 6.1.7 of \cite{NB}).\\


In view of the Nakano-Roberts Theorem, the next question would naturally be whether the hypothesis of Theorem \ref{theorem2.2} is sufficient to imply that $\mathcal{L}_{tc}(E_1, E_2)$ is an ideal of $\mathcal{L}_b(E_1, E_2)$. The answer is still negative as we can see from the next example.\\

\noindent \textbf{Counterexample 3.5--The hypothesis of Theorem~\ref{theorem2.2} is not sufficient to imply that the space of topologically continuous operators is an ideal of the space of the order bounded operators.} \\
\noindent Let $D[0, 2\pi]$ denote the family of all differentiable functions on $[0, 2\pi]$. Put $E_1=E_2=D[0, 2\pi]\times D[0, 2\pi]$ with $\tau_1=\tau_2$ being the norm topology generated by $||(x_1, x_2)||=||x_1||_{\infty}+ ||x_2 ||_{\infty}$, where $||x(t)||_{\infty}=\sup_{0\leq t\leq 2\pi} |x(t)|$ is the sup norm. Equip $E_1$ with the ordering defined as follows: for any two points $x=(x_1, x_2)$ and $y=(y_1, y_2)$ in $E_1$ we define $x\leq y$ if and only if $x_1(t)\leq y_1(t)$ and $x_2(t)\leq y_2(t)$ for all $t\in [0, 2\pi]$. Then $(E_1, \tau_1)$ is a locally-convex solid Riesz space.
For $E_2$, we equip it with the following ordering: for any two points $x=(x_1, x_2)$ and $y=(y_1, y_2)$ in $E_2$ we define $x\leq y$ if and only if $x_1(t) < y_1(t)$ for all $t\in [0, 2\pi]$ or else $x_1(t)=y_1(t)$ and $x_2(t)\leq y_2(t)$ for all $t\in [0, 2\pi]$.  It is  clear that $E_2$ has an order-bounded $\tau_2$-neighborhood of zero. Consider the operators $S$ and $T$ from $E_1$ to $E_2$ defined by
\begin{eqnarray*}
&S:& (x_1, x_2)\mapsto (0, x_2');\\
&T:& (x_1, x_2)\mapsto (x_2, 0).
\end{eqnarray*}
where $x'$ denotes the derivative of $x$. Then $T$ is order bounded and topologically continuous, $|S|\leq |T|$, and
$S$ is order bounded.
However, $S$ is not topologically continuous. To see this, consider the sequence $\{(0, (\sin kt)/k)\}_{k\geq 1}$ in $E_1$. We have $(0, (\sin kt)/k) \xrightarrow{\tau_1} 0$; but $S((0, (\sin kt)/k))=(0, \cos kt)\not \xrightarrow{\tau_2} 0$. Thus,  $\mathcal{L}_{tc}(E_1, E_2)$ is not an ideal of $\mathcal{L}_b(E_1, E_2)$.\\

\section{Relationship between $\mathcal{L}_{b}(E_1, E_2)$, $\mathcal{L}_{tb}(E_1, E_2)$ and
$\mathcal{L}_{tc}(E_1, E_2)$.}
In this section, we further investigate the relationship between $\mathcal{L}_{b}(E_1, E_2)$, $\mathcal{L}_{tb}(E_1, E_2)$ and $\mathcal{L}_{tc}(E_1, E_2)$. Example 2.8 of \cite{Hong} shows that an order bounded operator need not be topologically bounded, that is, $\mathcal{L}_{b}(E_1, E_2)\subset \mathcal{L}_{tb}(E_1, E_2)$ in general does not hold. On the other hand, Example 2.7 of \cite{Hong} shows that $\mathcal{L}_{tb}(E_1, E_2)\subset \mathcal{L}_{b}(E_1, E_2)$ in general does  not hold either. The following theorem gives sufficient conditions for the two spaces to coincide, i.e., $\mathcal{L}_{b}(E_1, E_2)=\mathcal{L}_{tb}(E_1, E_2)$.

\begin{theorem}\label{theorem5.1}
For $i=1,2$, let $(E_i, \tau_i)$ be a locally solid Riesz space.
\begin{enumerate}
  \item [(i)]If for $i=1, 2$, the space $(E_i, \tau_i)$ has an order bounded $\tau_i$-neighborhood of zero, then the space of order bounded operators coincide with the space of locally bounded operators, that is, $\mathcal{L}_{b}(E_1, E_2)=\mathcal{L}_{tb}(E_1, E_2)$.
  \item [(ii)]If the hypothesis of (i) holds and $(E_2, \tau_2)$ is also Dedekind complete, then $\mathcal{L}_{b}(E_1, E_2)=\mathcal{L}_{r}(E_1, E_2)=\mathcal{L}_{tb}(E_1, E_2)$.
  \item [(iii)]If for $i=1, 2$, $(E_i, \tau_i)$ has a nonempty interior of the positive cone $E_i^+$, then $\mathcal{L}_{b}(E_1, E_2)=\mathcal{L}_{tb}(E_1, E_2)$.
  \item [(iv)]If the hypothesis of (iii) holds and $E_2$ is also Dedekind complete, then $\mathcal{L}_{b}(E_1, E_2)=
      \mathcal{L}_{r}(E_1, E_2)=\mathcal{L}_{tb}(E_1, E_2)$.
\end{enumerate}
\end{theorem}

\begin{proof}
\begin{enumerate}
  \item [(i)]Take any $T\in \mathcal{L}_b(E_1, E_2)$. Let $A$ be a $\tau_1$-bounded subset of $E_1$. By Theorem~\ref{theorem1.1} and Theorem~\ref{theorem1.2}, $T(A)$ is $\tau_2$-bounded.
      Therefore,  $\mathcal{L}_{b}(E_1, E_2)\subset \mathcal{L}_{tb}(E_1, E_2)$.  On the other hand, take any $T\in \mathcal{L}_{tb}(E_1, E_2)$. Let $B$ be an order bounded subset of $E_1$. Then Theorem \ref{theorem1.1} and Theorem \ref{theorem1.2} imply that $T(B)$ is order bounded.
      Therefore,  $\mathcal{L}_{tb}(E_1, E_2)\subset \mathcal{L}_{b}(E_1, E_2)$.
  \item [(ii)]Since $E_2$ is Dedekind complete, we have $\mathcal{L}_{b}(E_1, E_2)=\mathcal{L}_r(E_1, E_2)$.
  \item [(iii)]Take an element $x$ from the interior of $E_1^+$ and a circled neighborhood $V$ of zero such that $x+V\subset E_1^+$.
                Then $x-V\subset E_1^+$. Hence, $V$ is contained in the order interval $[-x, x]$, implying that $[-x, x]$ itself is a $\tau_1$-neighborhood of zero. Similarly, $E_2$ has an order bounded $\tau_2$-neighborhood.
                Therefore, the conclusion follows from (i).
  \item [(iv)]Similar to (ii).
\end{enumerate}
\end{proof}

Next, we give conditions for the three spaces $\mathcal{L}_{b}(E_1, E_2)$, $\mathcal{L}_{tb}(E_1, E_2)$ and $\mathcal{L}_{tc}(E_1, E_2)$ to coincide.

\begin{theorem}\label{theorem5.2}
Suppose for $i=1, 2$, the space $(E_i, \tau_i)$ is a locally solid Riesz space having an order bounded $\tau_i$-neighborhood of zero. If any of the following conditions is satisfied, then $\mathcal{L}_{b}(E_1, E_2)=\mathcal{L}_{tb}(E_1, E_2)=\mathcal{L}_{tc}(E_1, E_2)$.

\begin{enumerate}
  \item [(i)]$(E_1, \tau_1)$ has a countable neighborhood base at zero.
  \item [(ii)]For $i=1, 2$, $(E_i, \tau_i)$ has a $\tau_i$-bounded convex neighborhood of zero.
  \item [(iii)]For $i=1, 2$, $(E_i, \tau_i)$ is a locally convex Hausdorff space.
  \item [(iv)]$(E_1, \tau_1)$ is bornological and $(E_2, \tau_2)$ is locally convex.
\end{enumerate}
\end{theorem}

\begin{proof}
In view of Theorem~\ref{theorem2.2} and Theorem~\ref{theorem5.1}, it suffices to establish that $\mathcal{L}_{tb}(E_1, E_2)\subset\mathcal{L}_{tc}(E_1, E_2)$.
\begin{enumerate}
  \item [(i)]Since $(E_1, \tau_1)$ has a countable neighborhood base at zero, it is pseudometrizable, implying  $\mathcal{L}_{tb}(E_1, E_2)\subset\mathcal{L}_{tc}(E_1, E_2)$.

  \item [(ii)]A topological vector space with a topologically bounded convex neighborhood of zero is
             seminormed. Therefore, $(E_1, \tau_1)$ and $(E_2, \tau_2)$ are both seminormable, implying $\mathcal{L}_{tb}(E_1, E_2)=\mathcal{L}_{tc}(E_1, E_2)$.
  \item [(iii)]A locally Hausdorff space is normable if and only if it has a topologically bounded neighborhood of zero.
               It follows that $(E_1, \tau_1)$ and $(E_2, \tau_2)$ are both normable. Hence, $\mathcal{L}_{tb}(E_1, E_2)=\mathcal{L}_{tc}(E_1, E_2)$.
  \item [(iv)]This follows from the fact that the class of bornological spaces is the class of locally convex spaces $(E_1, \tau_1)$ such that
               all topologically bounded operators from $(E_1, \tau_1)$ to a locally convex space $(E_2, \tau_2)$
               are topologically continuous. Therefore, we also have $ \mathcal{L}_{tb}(E_1, E_2)\subset\mathcal{L}_{tc}(E_1, E_2)$ in this case.
\end{enumerate}
\end{proof}

\begin{corollary}\label{corollary3.1}
For $i=1,2$, let $(E_i, \tau_i)$ be a locally solid Riesz space. Suppose any of the four conditions in Theorem \ref{theorem5.2} holds.
\begin{enumerate}
  \item [(i)]If for $i=1, 2$, $(E_i, \tau_i)$ has an order bounded $\tau_i$-neighborhood of zero and $E_2$ is also Dedekind complete, then $\mathcal{L}_{r}(E_1, E_2)=\mathcal{L}_r(E_1, E_2)=\mathcal{L}_{tb}(E_1, E_2)=\mathcal{L}_{tc}(E_1, E_2)$.
  \item [(ii)]If for $i=1, 2$, $(E_i, \tau_i)$ has a nonempty interior of the positive cone $E_i^+$, then $\mathcal{L}_{b}(E_1, E_2)=\mathcal{L}_{tb}(E_1, E_2)=\mathcal{L}_{tc}(E_1, E_2)$.
  \item[(iii)]If the hypothesis of (ii) holds and $E_2$ is also Dedekind complete, then $\mathcal{L}_{b}(E_1, E_2)=\mathcal{L}_r(E_1, E_2)=\mathcal{L}_{tb}(E_1, E_2)=\mathcal{L}_{tc}(E_1, E_2)$.
\end{enumerate}
\end{corollary}

\noindent \textbf{Remark 1.} If the hypothesis of Corollary \ref{corollary3.1} holds, then we have several interesting results: (i) every topologically bounded operator can be written as the difference of two positive operators; (ii) every positive operator is topologically continuous, and so on. \\

\noindent \textbf{Remark 2.} Let $\mathcal{L}_{n}(E_1, E_2)$ denote the space of all order continuous operators between $E_1$ and $E_2$. Then $\mathcal{L}_{n}(E_1, E_2)$ is a band of $\mathcal{L}_{b}(E_1, E_2)$ when $E_2$ is Dedekind complete and the Riesz-Kantorovich Theorem implies
\begin{equation*}
\mathcal{L}_b(E_1, E_2)=\mathcal{L}_{n}(E_1, E_2)\oplus \mathcal{L}_{n}^d(E_1, E_2).
\end{equation*}
Therefore, if the hypothesis of Corollary \ref{corollary3.1} holds, then every topologically bounded or topologically continuous operator can be decomposed as the sum of an order continuous operator and an operator in $\mathcal{L}_{n}^d(E_1, E_2)$. \\

\noindent \textbf{Remark 3.} Similar to Remark 2, when $E_2$ is Dedekind complete, we have
\begin{equation*}
\mathcal{L}_{b}(E_1, E_2)=\mathcal{L}_{c}(E_1, E_2)\oplus \mathcal{L}_{s}(E_1, E_2),
\end{equation*}
where $\mathcal{L}_{c}(E_1, E_2)$ denotes the space of all $\sigma$-order continuous operators and
$\mathcal{L}_{s}(E_1, E_2)$ denotes the space of all singular operators.
It follows that if the hypothesis of Corollary \ref{corollary3.1} holds, then every topologically bounded or topologically continuous operator can be decomposed as the sum of a $\sigma$-order continuous operator and a singular operator.

\section*{Acknowledgments}
Thanks are due to the Associate Editor and the anonymous reviewer for thoughtful comments and suggestions that led to significant improvements in this article.  The author also thanks Professor Gerard Buskes for his encouragement. \\

\bibliographystyle{amsplain}

\end{document}